\documentclass[10pt]{article}%
\usepackage{amsmath}
\usepackage{amsfonts}
\usepackage{amssymb}
\usepackage{graphicx}%
\providecommand{\U}[1]{\protect\rule{.1in}{.1in}}
\newtheorem{theorem}{Theorem}

\newtheorem{problem}[theorem]{Problem}

\newenvironment{proof}[1][Proof]{\noindent\textbf{#1.} }{\ \rule{0.5em}{0.5em}}
\begin{document}

\title{An application of Kapteyn series to a problem from queueing theory}
\author{Diego Dominici}
\date{}
\maketitle

\section{Introduction}

In \cite[4.47]{MR1610068}, the authors considered the following problem:

\begin{problem}
Find the unique solution $C(D)$ in the range $(0,1)$ of the transcendental
equation%
\begin{equation}
1=\frac{D}{2\left(  D+1\right)  }F(C), \label{eqC}%
\end{equation}
where
\begin{equation}
F(C)=%
{\displaystyle\sum\limits_{n=-\infty}^{\infty}}
J_{n}\left(  \frac{n}{\sqrt{D+1}}\right)  C^{n}, \label{F}%
\end{equation}
$J_{n}(\cdot)$ is the Bessel function of the first kind and $D>0.$
\end{problem}

They also defined the functions $C_{1}\left(  D,a\right)  $ and $C_{2}\left(
D,a\right)  $ implicitly by%
\begin{equation}
0=\frac{a}{2\sqrt{D+1}}F(C)+\frac{1}{2}C_{1}F_{1}(C), \label{eqC1}%
\end{equation}%
\begin{equation}
0=\frac{1}{4\sqrt{D+1}}\left[  1-\frac{a^{2}}{2\left(  D+1\right)  }\right]
F_{2}(C)+\left[  \frac{1}{4}\frac{a}{\sqrt{D+1}}C_{1}+C_{2}\right]  F_{1}(C),
\label{eqC2}%
\end{equation}
where $a>0$ and%
\begin{equation}
F_{1}(C)=%
{\displaystyle\sum\limits_{n=-\infty}^{\infty}}
nJ_{n}\left(  \frac{n}{\sqrt{D+1}}\right)  C^{n},\quad F_{2}(C)=%
{\displaystyle\sum\limits_{n=-\infty}^{\infty}}
nJ_{n}^{\prime}\left(  \frac{n}{\sqrt{D+1}}\right)  C^{n}. \label{F12}%
\end{equation}
Using some properties and asymptotic approximations for the Bessel functions,
they proved that%
\begin{equation}
\left(  \sqrt{D+1}-\sqrt{D}\right)  \exp\left(  \sqrt{\frac{D}{D+1}}\right)
<C(D)<1, \label{bound}%
\end{equation}%
\begin{equation}
C(D)\sim1-\frac{D^{2}}{4},\quad C_{1}(D)\sim\frac{aD}{2},\quad C_{2}%
(D)\sim\frac{1}{4}-\frac{a^{2}}{8},\quad D\rightarrow0,\quad C(D)\sim
\sqrt{\frac{e}{D}},\quad D\rightarrow\infty. \label{asympt}%
\end{equation}

Series of the form (\ref{F}), (\ref{F12}) are called Kapteyn series
\cite{MR1986919}. The purpose of this work is to obtain exact solutions of
(\ref{eqC}), (\ref{eqC1}) and (\ref{eqC2}) using properties of such series.

\section{Main Result}

\begin{theorem}
The solutions of (\ref{eqC}), (\ref{eqC1}) and (\ref{eqC2}) are given by%
\begin{equation}
C(D)=\frac{\exp\left(  \frac{1}{2}\frac{D}{D+1}\right)  }{\sqrt{D+1}},\quad
C_{1}(D)=\frac{D}{2\left(  D+1\right)  ^{\frac{3}{2}}}a,\quad C_{2}%
(D)=\frac{\sqrt{D+1}}{4}-\frac{D+\left(  D+1\right)  ^{\frac{3}{2}}}{8\left(
D+1\right)  ^{2}}a^{2}. \label{solution}%
\end{equation}

\end{theorem}

\begin{proof}
Let
\begin{equation}
\varepsilon=\frac{1}{\sqrt{D+1}},\quad C=e^{\mathrm{i}M},\quad M=E-\varepsilon
\sin(E).\label{eM}%
\end{equation}
Using (\ref{eM}) and the formula \cite[2.1 (2)]{MR1349110},
\begin{equation}
J_{-n}(z)=\left(  -1\right)  ^{n}J_{n}(z),\label{J-}%
\end{equation}
we can rewrite (\ref{eqC}) as%
\begin{equation}
\frac{2}{1-\varepsilon^{2}}=1+2%
{\displaystyle\sum\limits_{n=1}^{\infty}}
J_{n}\left(  n\varepsilon\right)  \cos(nM)=\frac{1}{1-\varepsilon\cos
(E)},\label{sum1}%
\end{equation}
where we have used \cite[17.21 (6)]{MR1349110} and%
\begin{equation}
r=\frac{\alpha\left(  1-\varepsilon^{2}\right)  }{1+\varepsilon\cos\left(
\omega\right)  }=\alpha\left[  1-\varepsilon\cos\left(  E\right)  \right]
.\label{r}%
\end{equation}
It follows from (\ref{sum1}) that%
\begin{equation}
E=\arccos\left(  \frac{1+\varepsilon^{2}}{2\varepsilon}\right)  \label{E}%
\end{equation}
and therefore%
\[
M=\arccos\left(  \frac{1+\varepsilon^{2}}{2\varepsilon}\right)  +\frac
{\mathrm{i}}{2}\left(  \varepsilon^{2}-1\right)  .
\]
Thus,%
\[
C(D)=e^{\mathrm{i}M}=\varepsilon\exp\left(  \frac{1-\varepsilon^{2}}%
{2}\right)  =\frac{\exp\left(  \frac{1}{2}\frac{D}{D+1}\right)  }{\sqrt{D+1}}.
\]

Using (\ref{eM}) and (\ref{J-}) in (\ref{F12}), we have \cite[17.21
(9-10)]{MR1349110}
\begin{align}
F_{1}(C)  &  =2\mathrm{i}%
{\displaystyle\sum\limits_{n=1}^{\infty}}
nJ_{n}\left(  n\varepsilon\right)  \sin(nM)=\frac{\alpha^{2}}{r^{2}}%
\sin\left(  \omega\right)  \frac{\varepsilon}{\sqrt{1-\varepsilon^{2}}%
}\mathrm{i},\label{F11}\\
F_{2}(C)  &  =2%
{\displaystyle\sum\limits_{n=1}^{\infty}}
nJ_{n}^{\prime}\left(  n\varepsilon\right)  \cos(nM)=\frac{\alpha^{2}}{r^{2}%
}\cos\left(  \omega\right)  ,\nonumber
\end{align}
where \cite[17.2 (1-3)]{MR1349110}%
\begin{equation}
\sin\left(  \omega\right)  =\frac{\sqrt{1-\varepsilon^{2}}}{1-\varepsilon
\cos\left(  E\right)  }\sin(E). \label{w}%
\end{equation}
Using (\ref{r}), (\ref{E}) and (\ref{w}) in (\ref{F11}), we obtain%
\begin{equation}
F_{1}(C)=-\frac{4}{\left(  1-\varepsilon^{2}\right)  ^{2}}=-4\left(
\frac{D+1}{D}\right)  ^{2},\quad F_{2}(C)=\frac{4}{\varepsilon\left(
1-\varepsilon^{2}\right)  ^{2}}=4\frac{\left(  D+1\right)  ^{\frac{5}{2}}%
}{D^{2}}. \label{Fexact}%
\end{equation}
Replacing (\ref{Fexact}) in (\ref{eqC1})-(\ref{eqC2}), the result follows.
\end{proof}

An easy computation, shows that (\ref{solution}) implies (\ref{bound}) and
(\ref{asympt}).

\bibliographystyle{abbrv}

\end{document}